\documentclass[12pt]{amsart}
\usepackage{amscd}
\usepackage{amssymb,latexsym}

\usepackage[pdftex,colorlinks=true]{hyperref}
\usepackage[T1]{fontenc}
\usepackage{XCharter}\usepackage{eulervm}
\usepackage{amsaddr}
\usepackage[all,cmtip]{xy}
\usepackage{graphicx}

\theoremstyle{plain}
\newtheorem{thm}{Theorem}[section]

\newtheorem*{thma}{Theorem A}
\newtheorem*{thmb}{Theorem B}
\newtheorem{lem}[thm]{Lemma}
\newtheorem{prop}[thm]{Proposition}
\newtheorem{cor}[thm]{Corollary}

\theoremstyle{definition}
\newtheorem{defn}[thm]{Definition}

\theoremstyle{remark}
\newtheorem{rem}[thm]{Remark}
\newtheorem*{ex}{Example}

\numberwithin{equation}{section}

\newcommand{\wt}{\widetilde}

\newcommand{\R}{\mathbb{R}}
\newcommand{\N}{\mathbb{N}}
\newcommand{\Z}{\mathbb{Z}}
\newcommand{\cC}{\mathcal{C}}
\newcommand{\G}{\mathcal{G}}
\newcommand{\restrictionmap}[2]{{#1}\mathpunct\restriction\hbox{}_{#2}}
\providecommand{\abs}[1]{\left\lvert#1\right\rvert}

\DeclareMathOperator{\id}{Id}
\DeclareMathOperator{\cl}{cl}

\DeclareMathOperator{\Span}{span}
\DeclareMathOperator{\Deg}{Deg}

\title[Otopy classification in Hilbert space]
{Otopy classification of gradient compact perturbations of identity in Hilbert space} 

\author[P. Bart{\l}omiejczyk and P. Nowak-Przygodzki]{Piotr Bart{\l}omiejczyk
and Piotr Nowak-Przygodzki}

\address{Faculty of Applied Physics and Mathematics,
Gda{\'n}sk University of Technology,
Gabriela Narutowicza 11/12,
80-233 Gda{\'{n}}sk, Poland}

\email{piobartl@pg.edu.pl, piotrnp@wp.pl}

\date{\today}

\begin{document}

\begin{abstract}
We prove that the inclusion of the space 
of gradient local maps into the space of
all local maps from Hilbert space to itself
induces a bijection between the sets of 
the respective otopy classes of these maps,
where by a local map we mean a compact perturbation 
of identity with a compact preimage of zero.
\end{abstract}

\maketitle

\section*{Introduction} 
\label{sec:intro}

In 1985 E. N. Dancer (\cite{D}) discovered that there is a better
topological invariant than the equivariant degree for gradient maps 
in the case of $S^1$ group action, which means that in that case
there are more equivariant gradient homotopy classes than
equivariant homotopy ones. 
A few years later A. Parusi\'nski (\cite{P})
showed that for a disc without group action there is 
no better invariant for gradient maps than the usual topological degree.
In other words, there is a bijection between sets 
of gradient and continuous homotopy classes. 
In 2005 E. N. Dancer, K. G\k{e}ba,
S. Rybicki (\cite{DGR}) provided the homotopy 
classification of equivariant gradient maps on the disc 
in the case of a compact Lie group action.
In their proof the authors used the notion of otopy introduced
in the 1990's by J. C. Becker and D. H. Gottlieb (\cite{BG1,BG2}).
Later, in \cite{BGI,GI} the equivariant and equivariant 
gradient otopy classifications instead of homotopy ones were studied.

The investigations mentioned above suggest the following general approach.
Let us consider vector fields on open domains contained
in some Riemannian manifold $X$ (in some cases equipped with
an action of a compact Lie group). If the set of zeros
of such a vector field is compact, we call them \emph{local maps}.
We introduce the following notation. Let
\begin{itemize}
	\item $\mathcal{C}(X)$ ($\mathcal{G}(X)$) be the set of continuous
	(gradient) local maps, 
	\item $\mathcal{C}[X]$ ($\mathcal{G}[X]$) be the set of usual
	(gradient) otopy classes of continuous (gradient) local maps, 
	\item $\iota\colon\mathcal{G}[X]\to\mathcal{C}[X]$ be the function
	between the respective otopy classes induced by the inclusion
	$\mathcal{G}(X)\hookrightarrow\mathcal{C}(X)$.
\end{itemize}
We will say that the inclusion  $\mathcal{G}(X)\hookrightarrow\mathcal{C}(X)$
has the \emph{Parusi\'nski property} if
$\iota\colon\mathcal{G}[X]\to\mathcal{C}[X]$ is bijective.
In a series of papers \cite{BP1,BP2,BP3,BP4}
we proved that the respective inclusions have
the Parusi\'nski property if $X$ is an open subset of $\R^n$
or, more generally, a Riemannian manifold 
(not necessarily compact) without boundary.
On the other hand, in \cite{BP5} we showed that
for an open invariant subset of 
a finite-dimensional representation
of a compact Lie group the inclusion $\iota$ does not 
have the Parusi\'nski property in general.
Moreover, we gave in that case necessary and sufficient conditions
for the Parusi\'nski property in terms of Weyl group dimensions,
which explains the phenomenon discovered by E. N. Dancer in 1985.

The presented paper is a natural continuation of our previous work.
Namely, the main aim of this article is to prove that the inclusion
of the space of gradient local maps into the space
of all local maps has the Parusi\'nski property
if $X$ is an open subset of a real 
separable Hilbert space. By \emph{local map} we mean here
a compact perturbation of identity with a compact
preimage of zero. It is worth pointing out that in the proof
of our main theorem we use a topological invariant,
which is a version of the classical Leray-Schauder degree.
But in our construction we manage to guarantee
that finite-dimensional approximations of gradient maps are gradient,
which is crucial for the proof of Theorem~A.
The results presented here may also be treated
as an introduction to the study of Parusi\'nski property
for a representation of a compact Lie group G in a Hilbert space. 

The organization of the paper is as follows.
Section~\ref{sec:basic} contains some preliminaries.
In Section~\ref{sec:deg} we describe a construction
of the topological degree used in the proof of Theorem A.
Our main results are stated in Section~\ref{sec:main}
and proved in Section~\ref{sec:proof}. 
Final remarks are contained in Section~\ref{sec:final}. 
Finally, Appendix~\ref{sec:appendixA} presents technical facts
used in Section~\ref{sec:proof}. 

\section{Basic definitions} 
\label{sec:basic}

Assume that
\begin{itemize}
	\item $E$ is an infinite-dimensional real separable Hilbert space,
	\item $\Omega$ is an open connected subset of $E$.
\end{itemize}
Recall that a continuous map from a metric space $A$ 
into a metric space $B$ is called \emph{compact} if it
takes bounded subsets of $A$ into relatively compact ones of $B$.
Some authors use the term \emph{completely continuous} 
instead of compact.

\subsection{Local maps in Hilbert space}
We write $f\in\cC(\Omega)$ if
\begin{itemize}
  \item $f\colon D_f\subset\Omega\to E$,
	\item $D_f$ is an open subset of $\Omega$,
	\item $f(x)=x-F(x)$, where $F\colon D_f\to E$ is compact,
	\item $f^{-1}(0)$ is compact.
\end{itemize}
Elements of $\cC(\Omega)$ are called \emph{local maps}.

It is easy to check that the compactness of $F$ implies that
in the above definition the last condition
that $f^{-1}(0)$ is compact can be equivalently replaced
by the assumption that $f^{-1}(0)$ 
is bounded and closed in $E$.
From this observation follows that if $f$ is defined on $\cl U$,
where $U$ is open and bounded, and $f$ does not vanish on the boundary
then $\restrictionmap{f}{U}$ is a local map.

Moreover, we write $f\in\G(\Omega)$ if 
\begin{itemize}
	\item $f\in\cC(\Omega)$,
	\item $f$ is gradient i.e.\ there is a $C^1$-function
	$\varphi\colon D_f\to\R$ such that $f=\nabla\varphi$.
\end{itemize}
Elements of $\G(\Omega)$ are called
\emph{gradient local maps}.

\subsection{Sets of otopy classes in Hilbert space}
A map $h\colon\Lambda\subset I\times\Omega\to E$ 
is called an \emph{otopy} if
\begin{itemize}
	\item $\Lambda$ is an open subset of $I\times\Omega$,
	\item $h(t,x)=x-F(t,x)$, where $F\colon\Lambda\to E$ is compact,
	\item $h^{-1}(0)$ is compact.
\end{itemize}

Similarly as in the case of local maps,
the assumption that the set $h^{-1}(0)$ is bounded
and closed in $I\times E$ implies its compactness.
In particular, if $\Lambda\subset I\times\Omega$ is 
open and bounded, $h$ is defined on $\cl\Lambda$ 
(not only on $\Lambda$) and
does not vanish on $\partial\Lambda$ then
$\restrictionmap{h}{\Lambda}$ is an otopy.

From the above and an easy to check fact
that a straight-line homotopy between 
two compact maps is compact we obtain the following result. 

\begin{lem}\label{lem:bound}
Assume that $U\subset E$ is open and bounded and
$h\colon I\times\cl U\to E$ is a straight-line homotopy.
If $h(t,x)\neq 0$ for $t\in I$ and $x\in\partial U$
and $\restrictionmap{h_0}{U}$ and $\restrictionmap{h_1}{U}$
are local maps, then $\restrictionmap{h}{I\times U}$ is an otopy.
\end{lem}

An otopy is called \emph{gradient}, if additionally
\[
F(t,x)=\nabla_x\eta(t,x),
\]
where $\eta\colon\Lambda\to\R$ 
is $C^1$ with respect to~$x$.

Given an otopy
$h\colon\Lambda\subset I\times\Omega\to E$ 
we can define for each  $t\in I$:
\begin{itemize}
	\item sets $\Lambda_t=\{x\in\Omega\mid(t,x)\in\Lambda\}$,
	\item maps $h_t\colon\Lambda_t\to E$ with $h_t(x)=h(t,x)$.
\end{itemize}
If $h$ is a (gradient) otopy, we can say that  
$h_0$ and $h_1$ are \emph{(gradient) otopic}.
Observe that (gradient) otopy establishes
an equivalence relation in $\cC(\Omega)$ ($\G(\Omega)$).
Sets of otopy classes of the respective relation
will be denoted by $\cC[\Omega]$ and $\G[\Omega]$.

Observe that if $f$ is a (gradient) local map and $U$ is 
an open subset of $D_f$ such that $f^{-1}(0)\subset U$, 
then $f$ and $\restrictionmap{f}{U}$ are (gradient) otopic.
This property of (gradient) local maps
will be called \emph{restriction property}.
In particular, if $f^{-1}(0)=\emptyset$ then $f$ 
is (gradient) otopic to the empty map.

\begin{rem}\label{rem:finite}
It is worth pointing out that in \cite{BP1,BP3} we consider
local maps and otopies in finite dimensional spaces.
Unlike as in the case of Hilbert space we assume in the definition
of both a local map and an otopy only the condition that
the preimage of zero is compact. There is no need to assume
the form $\id-F$ with $F$ compact. However, subsequently
in the proof of the main result of this paper we will need
the form `identity minus compact' in a finite dimensional case.
This will be guaranteed by boundedness of a domain of a map.
\end{rem}

\section{Definition of degree \texorpdfstring{$\Deg$}{Deg}} 
\label{sec:deg}

In this section we give a definition of the degree
{$\Deg\colon\cC(\Omega)\to\Z$
and prove its correctness and otopy invariance.

\subsection{Preparatory lemmas}
Let us start with the following lemma concerning $f\in\cC(\Omega)$.

\begin{lem}\label{lem:zero}
Assume that $X\subset D_f$ is closed in $E$ and bounded.
If $X\cap f^{-1}(0)=\emptyset$ then there is $\epsilon>0$ such that
$\abs{f(x)}\ge2\epsilon$ for all $x\in X$.
\end{lem}

\begin{proof}
Suppose that there is a sequence $\{x_n\}\subset X$
such that $\lim f(x_n)=0$. By compactness of $F$ there is 
a subsequence $\{x_{k_n}\}$ of $\{x_n\}$ 
such that $\lim F(x_{k_n})=y$
and therefore  $\lim x_{k_n}=y$. Since $X$ is closed,
we have $y\in X$ and, in consequence, $f(y)\neq0$.
But $f(y)=\lim f(x_{k_n})=0$, a contradiction.
\end{proof}

Observe that there is an open bounded set $U$ such that
\[
f^{-1}(0)\subset U\subset\cl U\subset D_f.
\]

\begin{cor}\label{cor:epsilon}
There is $\epsilon>0$ such that
$\abs{f(x)}\ge2\epsilon$ for all $x\in \partial U$.
\end{cor}

Let $\{e_i\mid i\in\N\}$ be an orthonormal basis in $E$.
Let us introduce the following notation for $n\in\N$:
\begin{itemize}
	\item $V_n=\Span\{e_1,\dotsc,e_n\}$,
	\item $U_n=U\cap V_n$ for any $U\subset E$,
	\item $f_n(x)=x-P_nF(x)$, where $P_n\colon E\to V_n$
	is an orthogonal projection.
\end{itemize}

Throughout the paper we will make use of the following
well-known characterization of relatively 
compact sets in Hilbert space.

\begin{prop}\label{prop:compact}
A set $X\subset E$ is relatively compact if{f} it is bounded and
\[
\forall\delta>0\,
\exists n_0\,
\forall n\ge n_0\,
\forall x\in X\,\abs{x-P_nx}<\delta.
\]
\end{prop}

Now we are in position to show that for $n$ large enough
$f$ and $f_n$ are close to each other on $\cl U$.
Next from this observation we conclude that $f_n$ 
are uniformly separated from $0$ on $\partial U$.
\begin{lem}\label{lem:N}
There is $N$ such that for all $n\ge N$ and all $x\in\cl U$
we have:
\[
\abs{f(x)-f_n(x)}<\epsilon\quad
\text{and consequently}\quad
\abs{f_{n+1}(x)-f_n(x)}<\epsilon.
\]
\end{lem}

\begin{proof}
Since $\cl U$ is bounded,
$F(\cl U)$ is relatively compact.
By Proposition \ref{prop:compact}
there is $N$ such that for all $n\ge N$ 
and all $x\in\partial U$ we have
$\abs{F(x)-P_nF(x)}<\epsilon$.
Since $\abs{f(x)-f_n(x)}=\abs{F(x)-P_nF(x)}$,
we obtain our assertion.
\end{proof}

From now on let $N$ be chosen 
as in the previous lemma.

\begin{lem}\label{lem:boundary}
$\abs{f_n(x)}\ge\epsilon$
for $x\in\partial U$ and $n\ge N$.
\end{lem}
\begin{proof}
It is an easy consequence of
Corollary \ref{cor:epsilon} and 
Lemma \ref{lem:N}.
\end{proof}

\subsection{Definition of\ \texorpdfstring{$\Deg$}{Deg}}
In what follows, $\deg$ denotes  the classical Brouwer degree.
The infinite-dimensional degree that we are going to define
in this paper will be denoted by $\Deg$.

Since $\partial U_n\subset\partial U$ for any $n$,
the next result follows from Lemma~\ref{lem:boundary}. 

\begin{cor}\label{cor:well}
$\deg(f_n,U_n)$ is well-defined for $n\ge N$.
\end{cor}

The following fact shows that the sequence 
$\{\deg(f_n,U_n)\}_{n\ge N}$ is constant.

\begin{lem}\label{lem:well}
$\deg(f_{n+1},U_{n+1})=\deg(f_n,U_n)$ for $n\ge N$.
\end{lem}

\begin{proof}
Since $f_n^{-1}(0)\subset U_n$, there is an open 
subset $W\subset V_n$ such that
\[
f_n^{-1}(0)\subset W \subset\cl W
\subset U_n
\]
and there is $\delta>0$ such that 
$W_\delta:=W\times(-\delta,\delta)\subset U_{n+1}$.
Let $g_n\colon W_\delta\to E$ be given by
$g_n(x)=x-P_nF(P_n x)$ (in other words $g_n$
is a suspension of $\restrictionmap{f_n}{W}$).
By definition, 
$\restrictionmap{g_n}{W}=\restrictionmap{f_n}{W}$
and $g_n^{-1}(0)=f_n^{-1}(0)$.
Let us check the following sequence of equalities.
\begin{multline*}
\deg(f_{n+1},U_{n+1})\stackrel{(1)}{=}
\deg(f_{n},U_{n+1})\stackrel{(2)}{=}
\deg(f_{n},W_\delta)\stackrel{(3)}{=}\\
\deg(g_{n},W_\delta)\stackrel{(4)}{=}
\deg(f_{n},W)\stackrel{(5)}{=}
\deg(f_{n},U_{n})
\end{multline*}
The equalities $(1)$ and $(3)$ can be obtained
using straight-line homotopies,
which are otopies by Lemmas \ref{lem:N} and \ref{lem:boundary}.
In turn $(2)$ and $(5)$ are based on the restriction
property of the degree and, finally,
$(4)$ follows from the fact that
$g_n$ is a suspension of $f_n$ over $W$.
This completes the proof.
\end{proof}

Lemma \ref{lem:well} guarantees that the following definition
does not depend on the choice of admissible $N$. 

\begin{defn}
Define $\Deg f=\Deg(f,U)=\deg(f_N,U_N)$.
\end{defn}

\begin{rem}
Our degree gives the same values as the classical
Leray-Schauder degree, which follows easily
from the comparison of our construction and
the definition and the proof of the well-definedness
of the Leray-Schauder degree. 
However, in the proof of bijectivity of our degree
in the gradient case we use the fact that 
finite-dimensional approximations of a gradient map
appearing in our construction are gradient,
which is not guaranteed by the original 
Leray-Schauder construction.
\end{rem}

\subsection{Correctness}
Let us note that for the above construction
we have chosen a neighbourhood $U$ of $f^{-1}(0)$ and
an orthonormal basis of $E$.
Now we are going to prove that our definition of $\Deg f$ 
does not depend on the choice of these both elements.

\begin{prop}
Let $W$ and $U$ be open bounded such that
\[
f^{-1}(0)\subset W\subset U\subset\cl U\subset D_f.
\]
Then $\Deg(f,W)=\Deg(f,U)$.
\end{prop}

\begin{proof}
By Lemma~\ref{lem:N}, we have $\abs{f(x)-f_n(x)}<\epsilon$
for $x\in\cl U$ and, consequently, $f_n(x)\neq0$ for
$x\in\cl U_n\setminus W_n\subset\cl U\setminus W$
for sufficiently large $n$. Hence
\[
\Deg(f,W)=\deg(f_n,W_n)=
\deg(f_n,U_n)=\Deg(f,U).\qedhere
\]
\end{proof}

\begin{cor}
Let $U$ and $U'$ be open bounded subsets of $D_f$ such that
\[
f^{-1}(0)\subset U\cap U'\subset
\cl(U\cap U')\subset\cl(U\cup U')\subset D_f.
\]
Then
\[
\Deg(f,U)=\Deg(f,U\cap U')=\Deg(f,U').
\]
\end{cor}

In this way we have proved that $\Deg f$ does not depend
on the choice of $U$.

In the remainder of this subsection 
we show that $\deg f$ does not depend
on the choice of an orthonormal basis in $E$.
The reasoning requires some additional notation.
Let $V$ be a finite dimensional linear subspace of $E$.
Set
\begin{itemize}
\item $U_V=U\cap V$,
\item $P_V\colon E\to V$ --- an orthogonal projection,
\item $f_V(x)=x-P_VF(x)$.
\end{itemize}
Analogously to Corollary \ref{cor:well} and Lemma \ref{lem:well}
one can prove the following result.

\begin{lem}\label{lem:stab}
If $V$ is a finite dimensional linear subspace of $E$
such that $V_N\subset V$ then $\deg(f_V,U_V)$ is well defined
and $\deg(f_N,U_N)=\deg(f_V,U_V)$.
\end{lem}

\begin{cor}
$\Deg f$ does not depend
on the choice of an orthonormal basis in $E$.
\end{cor}

\begin{proof}
Let  $\{e_i\}$ and $\{e'_i\}$ be two orthonormal bases in $E$.
We will use analogous notation for them both
writing prime where needed. For example,
$V_n=\Span\{e_1,\dotsc,e_n\}$ and $V'_n=\Span\{e'_1,\dotsc,e'_n\}$.
Let us choose $N$ and $N'$ for $\{e_i\}$ and $\{e'_i\}$
respectively as in Lemma \ref{lem:N}. Put $V=V_N+V'_{N'}$.
By Lemma \ref{lem:stab},
\[
\deg(f_N,U_N)=\deg(f_V,U_V)=\deg(f'_{N'},U'_{N'}),
\]
which is our assertion.
\end{proof}

\subsection{Otopy invariance of degree}

Let the map $h\colon\Lambda\subset I\times\Omega\to E$
given by $h(t,x)=x-F(t,x)$ be an otopy.
We introduce the following notation:
\begin{align*}
\Lambda^t=&\{x\in\Omega\mid(t,x)\in\Lambda\},
&h^t&\colon\Lambda^t\to E,
&h^t(x)&=h(t,x),\\
\Lambda_n=&\Lambda\cap(I\times V_n),
&h_n&\colon\Lambda\to V_n,
&h_n(t,x)&=x-P_nF(t,x),\\
\Lambda_n^t=&\Lambda^t\cap V_n,
&h_n^t&\colon\Lambda^t\to V_n,
&h_n^t(x)&=h_n(t,x).
\end{align*}
Note that for the needs of this subsection
the time parameter $t$ of otopy is
a superscript, not a subscript.

\begin{prop}[otopy invariance]
If $h\colon\Lambda\subset I\times\Omega\to E$ 
is an otopy then
\[
\Deg(h^0,\Lambda^0)=\Deg(h^1,\Lambda^1).
\]
\end{prop}

\begin{proof}
Since $h^{-1}(0)$ is compact,
there is an open bounded set
$W\subset I\times E$ such that
\begin{equation}\label{eqn:otopy1}
h^{-1}(0)\subset W\subset\cl W\subset\Lambda.
\end{equation}
Hence for $i=0,1$ we have
\begin{equation}\label{eqn:otopy2}
(h^i)(0)^{-1}\subset W^i\subset\cl W^i\subset\Lambda^i,
\end{equation}
where $W^i=\{x\in\Omega\mid(i,x)\in W\}$.
Analogously, as in Lemma~\ref{lem:zero}, from \eqref{eqn:otopy1}
there is $\epsilon>0$ such that $\abs{h(z)}\ge2\epsilon$
for $z\in\partial W$ and, as in Lemma~\ref{lem:N}, there is $N$
such that $\abs{h(z)-h_n(z)}<\epsilon$ for $n\ge N$
and $z\in\partial W$. Hence $\abs{h_n(z)}\ge\epsilon$ for
$z\in\partial W_n\subset\partial W$ and, in consequence,
\begin{itemize}
\item $\Deg(h^i,\Lambda^i)=\deg(h_n^i,W_n^i)$ for $i=0,1$,
\item $\restrictionmap{h_n}{W_n}$ is a finite-dimensional otopy,
\end{itemize}
which gives 
\[
\Deg(h^0,\Lambda^0)=\deg(h_n^0,W_n^0)=\deg(h_n^1,W_n^1)
=\Deg(h^1,\Lambda^1).\qedhere
\]
\end{proof}

\begin{rem}
Since our degree is otopy invariant,
it can be defined on the set of otopy class i.e. 
$\Deg\colon\cC[\Omega]\to\Z$. Moreover,
any gradient otopy class (as a set of functions)
is contained in a usual otopy class,
and hence the degree makes sense as a function
$\Deg\colon\G[\Omega]\to\Z$.
Without ambiguity we will use the symbol $\Deg$
in all the above cases.
\end{rem}

\section{Main results} 
\label{sec:main}

Let us formulate the main results of our paper.
\begin{thma}
The functions $\Deg\colon\cC[\Omega]\to\Z$ 
and $\Deg\colon\G[\Omega]\to\Z$
are bijections.
\end{thma}

It is obvious that the inclusion $\G(\Omega)\hookrightarrow\cC(\Omega)$
induces a well-defined function $\iota\colon\G[\Omega]\to\cC[\Omega]$.
The next result follows immediately from Theorem~A and 
the commutativity of the diagram
\begin{equation*}
\xymatrix{
\G[\Omega]\ar[rd]_\Deg\ar[rr]^\iota & &
\cC[\Omega]\ar[ld]^\Deg\\
&\Z.}
\end{equation*}

\begin{thmb}
The function $\iota\colon\G[\Omega]\to\cC[\Omega]$ is bijective.
\end{thmb}

\begin{rem}
In other words, there is no better invariant than
the Leray-Schauder degree that distinguishes
between two gradient local maps which are not gradient otopic.
\end{rem}

\section{Proof of Theorem A} 
\label{sec:proof}

\subsection{Injectivity of \texorpdfstring{$\Deg\colon\cC[\Omega]\to\Z$}{Deg}}\label{subsec:inj}
Let $f\colon D_f\to E$ and  $g\colon D_g\to E$ be local maps
such that $\Deg f=\Deg g$. We show that $f$ and $g$ are otopic.
The proof of that will be divided into two steps.
In the first step we show that $f$ is otopic to the suspension
of its finite dimensional approximation and in the second step
that suspensions of approximations for $f$ and $g$ 
are otopic one to another. We start with the observation 
that there exists an open bounded $U\subset E$ such that
\begin{itemize}
\item $f^{-1}(0)\subset U\subset\cl U\subset D_f$,
\item there is $N$ such that $P_n(\cl U)\subset D_f$
for all $n\ge N$.
\end{itemize}
The proof of this observation will be postponed 
to Appendix \ref{sec:appendixA}
(see Lemma \ref{lem:acompact}).

\vspace{1mm}\noindent\textbf{Step 1.}
For $n\ge N$ let $\Sigma f_n\colon\cl U\cup P_n^{-1}(U_n)\to E$
be given by $\Sigma f_n(x)=x-P_nF(P_nx)$.
Note that $\restrictionmap{\Sigma f_n}{P_n^{-1}(U_n)}$
is a suspension of $\restrictionmap{f_n}{U_n}$
(see Section \ref{sec:deg}).
We prove the following sequence of
otopy relations  for $n$ large enough:
\begin{equation}\label{eq:sim}
f\stackrel{(1)}{\sim}
\restrictionmap{f}{U}\stackrel{(2)}{\sim}
\restrictionmap{\Sigma f_n}{U}\stackrel{(3)}{\sim}
\restrictionmap{\Sigma f_n}{U\cup(P_n^{-1}(U_n)\cap\Omega)}
\stackrel{(4)}{\sim}
\restrictionmap{\Sigma f_n}{P_n^{-1}(U_n)\cap\Omega}.
\end{equation}
The sets appearing in \eqref{eq:sim} 
are shown in Figure \ref{fig:sets}.
First observe that all the maps in the above sequence
are local, because $(\Sigma f_n)^{-1}(0)\Subset U_n$ 
from Lemma \ref{lem:boundary}.
The relations $(1)$, $(3)$ and $(4)$ 
follow immediately from the restriction property.
To obtain $(2)$ let us consider the straight-line homotopy
$h_n\colon I\times\cl U\to E$ given by
$h_n(t,x)=(1-t)f(x)+t\Sigma f_n(x)$.
We show that there is $M\ge N$ such that $h_n(t,x)\neq0$
for $t\in I$, $x\in\partial U$ and $n\ge M$.
Thus $\restrictionmap{h_n}{I\times U}$
is an otopy, which proves the relation $(2)$.
On the contrary, suppose that there is an increasing
subsequence $\{n_k\}$ of natural numbers ($n_k\ge N$)
and sequences $\{t_k\}\subset I$ and $\{x_k\}\subset\partial U$
such that $h_{n_k}(t_k,x_k)=0$ i.e.\
\[
x_k=F(x_k)+t_k(P_{n_k}F(P_{n_k}x_k)-F(x_k)).
\]
By compactness of $F$ and $I$, we can assume that sequences
$t_k$, $F(x_k)$ and $F(P_{n_k}x_k)$ are convergent,
so $\{x_k\}$ is also convergent to some point $x_0\in\partial U$.
Since $x_k\to x_0$ implies $P_{n_k}x_k\to x_0$,
we obtain $f(x_0)=x_0-F(x_0)=0$, which contradicts
the fact that $f$ does not vanish on~$\partial U$.
\begin{figure}[ht]
\centering
\includegraphics[scale=1.4,trim= 55mm 205mm 45mm 45mm]{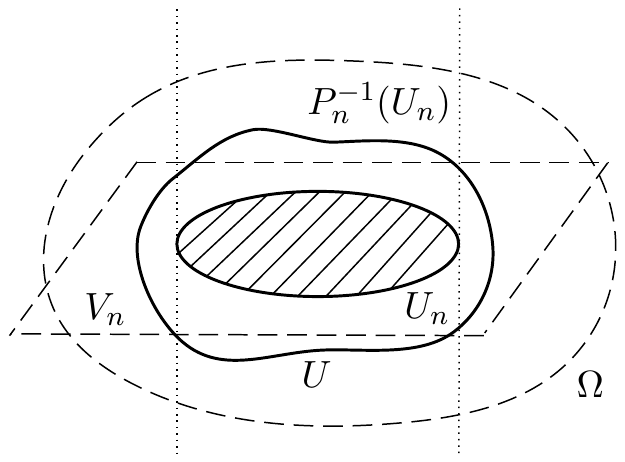}
\caption{Domains in \eqref{eq:sim}}
\label{fig:sets}
\end{figure}

\vspace{1mm}\noindent\textbf{Step 2.}
The same reasoning can be applied to the map $g$.
Similarly as for $f$ let us introduce the notation
$\Sigma g_n$ and a set $W\subset D_g$ 
(a counterpart of $U\subset D_f$).
We obtain in this way an analogical sequence of relations,
which gives 
$g\sim\restrictionmap{\Sigma g_n}{P_n^{-1}(W_n)\cap\Omega}$.
To finish the proof of injectivity it is enough
to show that
\[
\restrictionmap{\Sigma f_n}{P_n^{-1}(U_n)\cap\Omega}
\sim
\restrictionmap{\Sigma g_n}{P_n^{-1}(W_n)\cap\Omega}.
\]
To do that we will use Lemma~\ref{lem:aotopy},
which shows how to suspend finite dimensional otopies.

By the definition of our degree,
we have 
\[
\deg\restrictionmap{f_n}{U_n}=
\Deg f=\Deg g=\deg\restrictionmap{g_n}{W_n}.
\]
Unfortunately, this does not imply that
$\restrictionmap{f_n}{U_n}$ and
$\restrictionmap{g_n}{W_n}$ are finite dimensionally otopic,
since $U_n$ and $W_n$ may not be contained in the same
component of $\Omega_n$. Therefore, 
using Lemma~\ref{lem:aconnected}, the problem
will be lifted to a higher dimension, where
the relation of otopy holds.

Precisely, note first 
that $f_n^{-1}(0)\subset U_n$, $g_n^{-1}(0)\subset W_n$
and $K=f_n^{-1}(0)\cup g_n^{-1}(0)\subset V_n$
is compact. By Lemma~\ref{lem:aconnected},
$K$ is contained in one component of $\Omega_m$ 
for $m\ge n$ large enough. 
Let us denote this component by $\Omega'_m$.

Set $X=P_n^{-1}(U_n)\cap U\cap\Omega'_m$ and
$Y=P_n^{-1}(W_n)\cap W\cap\Omega'_m$.
Observe that 
\begin{itemize}
\item $X$ and $Y$ are open bounded,
\item $\cl X\subset\cl U\subset D_f$ and 
      $\cl Y\subset\cl W\subset D_g$,
\item $X\cup Y\subset\Omega'_m\subset V_m$.
\end{itemize}
Since $\deg\restrictionmap{f_n}{U_n}=
\deg\restrictionmap{g_n}{W_n}$, we have
$\deg\Sigma\restrictionmap{f_n}{X}=
\deg\Sigma\restrictionmap{g_n}{Y}$.
Moreover, the maps $\Sigma\restrictionmap{f_n}{X}$ and 
$\Sigma\restrictionmap{g_n}{Y}$ are bounded,
because $\Sigma f_n$ and $\Sigma g_n$ are defined 
on $\cl X$ and $\cl Y$ respectively.
Since $\Omega'_m$ is connected there is a bounded
finite dimensional otopy
$k\colon\Gamma\subset I\times\Omega'_m\to V_m$
between $\Sigma\restrictionmap{f_n}{X}$ and 
$\Sigma\restrictionmap{g_n}{Y}$ (see \cite[Rem. 2.3]{B}).
By Lemma~\ref{lem:aotopy}, there is an otopy in $\Omega$
between $\restrictionmap{\Sigma f_n}{P_m^{-1}(X)\cap\Omega}$ and 
$\restrictionmap{\Sigma g_n}{P_m^{-1}(Y)\cap\Omega}$.
Finally, since $P_m^{-1}(X)\subset P_m^{-1}(P_n^{-1}(U_n))
= P_n^{-1}(U_n)$ and similarly
$P_m^{-1}(Y)\subset P_n^{-1}(W_n)$, we obtain
\[
\restrictionmap{\Sigma f_n}{P_n^{-1}(U_n)\cap\Omega}
\sim\restrictionmap{\Sigma f_n}{P_m^{-1}(X)\cap\Omega}
\sim\restrictionmap{\Sigma g_n}{P_m^{-1}(Y)\cap\Omega}
\sim\restrictionmap{\Sigma g_n}{P_n^{-1}(W_n)\cap\Omega},
\] 
which completes the proof
of injectivity of $\Deg\colon\cC[\Omega]\to\Z$.

\subsection{Injectivity of \texorpdfstring{$\Deg\colon\G[\Omega]\to\Z$}{gradient Deg}}
Let $f,g\in\G(\Omega)$ and $\Deg f=\Deg g$.
To show that $[f]=[g]$ in $\G[\Omega]$
it is enough to observe that all otopies appearing
in the sequence connecting $f$ and $g$ as in \ref{subsec:inj}.
are in fact gradient. Namely
\begin{enumerate}
\item otopies connecting gradient local maps with their restrictions
are obviously gradient,
\item the straight-line  homotopy $(1-t)f+t\Sigma f_n$ is gradient,
because $f$ and $\Sigma f_n$ are gradient (note that if
$\varphi$ is a potential for $F$ then $\varphi\circ P_n$
is a potential for $P_nFP_n$), 
\item the otopy between 
$\restrictionmap{\Sigma f_n}{P_m^{-1}(X)\cap\Omega}$ and 
$\restrictionmap{\Sigma g_n}{P_m^{-1}(Y)\cap\Omega}$
appearing in Step 2 of \ref{subsec:inj} 
(see Lemma~\ref{lem:aotopy}) can be considered gradient,
because by Main Theorem in \cite[Sec. 2]{BP4}
$k(t,x)$ can be chosen gradient
(if $\varphi(t,x)$ is a family of potentials for $k(t,x)$
then we can take $\varphi(t,x)+\frac12\abs{y}^2$
as a family of potentials for for our otopy). 
\end{enumerate}

\subsection{Surjectivity of 
\texorpdfstring{$\Deg\colon\G[\Omega]\to\Z$}{gradient Deg}}\label{subsec:sur}
Using standard local maps (see Section 3 in \cite{BP1}) it is easy
to construct for any $m\in\Z$ a gradient local map
$f\colon D_f\subset V_n\cap\Omega\to V_n$ such that $\deg f=m$
($V_n\cap\Omega$ is nonempty for $n$ large enough).
Since as we observed suspensions of gradient local maps are 
also gradient local, the map $\Sigma f\colon P_n^{-1}(D_f)\cap\Omega\to E$
is an element of $\G[\Omega]$ and $\Deg \Sigma f=\deg f=m$.

\subsection{Surjectivity of 
\texorpdfstring{$\Deg\colon\cC[\Omega]\to\Z$}{Deg}}
Since any gradient local map is also a local map, 
it is an obvious consequence of~\ref{subsec:sur}.
\qed

\section{Final remarks} 
\label{sec:final}

This section is devoted to two possible directions
of developments of subject presented here.
Namely, we can additionally consider a group action and/or
linear operators other than identity.

\subsection{The case of a compact Lie group action}\label{subsec:Lie}
In \cite{BP5} we proved that for a finite dimensional
representation of a compact Lie group the function
induced on the sets on otopy classes by the inclusion 
of the set of equivariant gradient local maps into 
the set of equivariant local maps is a bijection
if and only if all Weyl groups appearing in 
the representation are finite. In consequence,
contrary to our Theorem B the function $\iota$
need not be bijective. We expect that an analogical
result holds for a Hilbert representation
of a compact Lie group. Here we will just give 
an example of two equivariant gradient local maps 
in Hilbert space that are otopic but not gradient otopic,
which illustrates that the function analogical
to $\iota$ in Theorem B may not be bijective.

\begin{ex}
Let $E'$ be a Hilbert space and $E=\mathbb{C}\oplus E'$.
Assume that $S^1$ acts on $\mathbb{C}\oplus E'$ 
by $g(z,x)=(gz,x)$. Consider for $i=0,1$ potentials 
$\varphi_i\colon\mathbb{C}\to\R$ given by
\[
\varphi_i(z)=\begin{cases}
(\abs{z}-1)^2& \text{if $\abs{z}\ge1$},\\
(1-2i)(\abs{z}-1)^2& \text{if $\abs{z}<1$}.
\end{cases}
\]
Set $f_i=\nabla\varphi_i$ (see Figure \ref{fig:concentric})
and $U=\{z\in\mathbb{C}\mid 1/2<\abs{z}<3/2\}$.
Let $\id$ denote the identity on $E'$.
Define $\wt{f}_i\colon U\times E'\to E$ 
by $\wt{f}_i=f_i\times\id$. 
It follows easily that $\wt{f}_0$ and $\wt{f}_1$
are equivariant otopic. We expect that it is possible
to show that they are not equivariant gradient otopic.
\end{ex}

\begin{figure}[ht]
\centering
\includegraphics[scale=0.8,trim= 0mm 0mm 0mm 0mm]{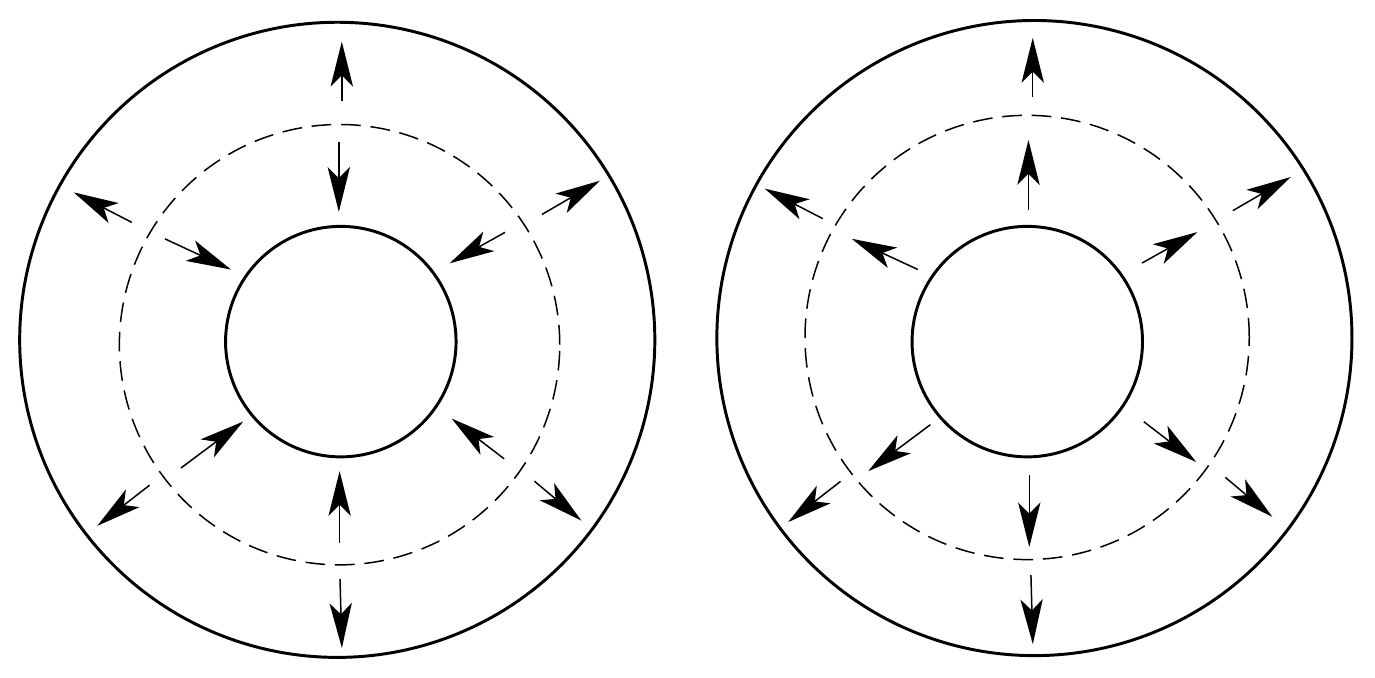}
\caption{Maps $f_0$ and $f_1$}
\label{fig:concentric}
\end{figure}

\subsection{The case of an unbounded operator}
In this paper we considered perturbations of 
the identity operator in Hilbert space.
Possible applications in Hamiltonian systems
and in the Seiberg-Witten theory suggest
replacing the identity by an unbounded 
self-adjoint operator with a purely discrete spectrum.
In that case in the absence of a group action
we expect the result similar to Theorem B.
However, if we take into account a group action
similarly as in Subsection \ref{subsec:Lie}
we may obtain the function $\iota$ that is not bijective.
This means that in the equivariant gradient case
we may get an extra topological invariant.

\appendix


\section{}\label{sec:appendixA}

In this appendix we have collected some technical
results needed in Section~\ref{sec:proof}.

\begin{lem}\label{lem:acompact}
Let $\Omega$ be an open subset of a separable Hilbert space $E$
and $K$ a compact subset of $\Omega$.
There exist an open bounded $U\subset E$
and natural number $N$ such that
\begin{itemize}
\item $K\subset U\subset\cl U\subset\Omega$,
\item $P_n(\cl U)\subset\Omega$ for all $n\ge N$.
\end{itemize}
\end{lem}

\begin{proof}
Let us denote by $B(x,R)$ the open ball
and  by $D(x,R)$ the closed ball in $E$
of radius $R>0$ centered at $x$.
Note that for any $x\in\Omega$ there is $R_x>0$ and $N_x\in\N$
such that $B(x,R_x)\subset\Omega$ and $\abs{P_nx-x}<R_x/2$
for $n\ge N_x$. If $\abs{y-x}\le R_x/2$ then
$\abs{P_ny-x}\le\abs{P_ny-P_nx}+\abs{P_nx-x}
<\abs{y-x}+R_x/2\le R_x$, and hence $P_ny\in B(x,R_x)\subset\Omega$.
In other words $P_n(D(x,R_x/2))\subset\Omega$ for $n\ge N_x$.
Since $K$ is compact, we can choose
$x_1,\dotsc,x_m\in K$ such that
$K\subset\bigcup_{i=1}^m B(x_i,R_{x_i}/2)$.
Set $U=\bigcup_{i=1}^m B(x_i,R_{x_i}/2)$ and 
$N=\max{\{N_{x_i}\mid i=1,\dotsc,m\}}$.
It is easy to see that
\[
P_n(\cl U)=
P_n\Bigl(\bigcup_{i=1}^m D(x_i,R_{x_i}/2)\Bigr)
\subset\Omega
\]
for $n\ge N$.
\end{proof}

\begin{cor}\label{cor:apath}
With the same notation and assumptions as above,
there is $N$ such that $P_n(K)\subset\Omega$
for $n\ge N$.
\end{cor}

\begin{rem}
The corollary is an immediate consequence of Lemma~\ref{lem:acompact},
but it can also be easily concluded from the characterization
of compact sets in $E$ (Prop. \ref{prop:compact}).
\end{rem}

\begin{lem}\label{lem:aconnected}
Let $\Omega$ be an open connected subset of $E$
and $K\subset\Omega_n:=\Omega\cap V_n$ be compact.
Then $K$ is contained in one component of $\Omega_m$
for $m$ large enough.
\end{lem}

\begin{proof}
Since $K$ is compact, it can be covered by a finite number
of balls $B_i\subset\Omega_n$. $\Omega$ is connected,
so there is a path $\omega_{ij}\subset\Omega$ from $B_i$ to $B_j$
for each pair  $i,j$. By Corollary \ref{cor:apath},
$P_{l_{ij}}(\omega_{ij})\subset\Omega_{l_{ij}}$ 
for $l_{ij}$ sufficiently large,
so all balls $B_i$ are contained in the same 
component of $\Omega_m$, where $m:=\max{\{l_{ij}\mid i,j\}}$.
\end{proof}

Let $\Omega_n:=\Omega\cap V_n$ and
$\Gamma\subset I\times \Omega_n$ are open.
Assume that $k\colon\Gamma\to V_n$ is a bounded
finite dimensional otopy
(see Remark \ref{rem:finite}). Let us define:
\begin{itemize}
\item $\Lambda=
\big(\Gamma\times V_n^{\bot}\big)\cap\big( I\times\Omega\big)$
\item $h\colon\Lambda\to E$ given by $h(t,x,y)=(k(t,x),y)$,
where $t\in I$, $x\in V_n$, $y\in V_n^{\bot}$
(note that $(t,x,y)\in\Lambda$ implies $(t,x)\in\Gamma$).
\end{itemize}

\begin{lem}\label{lem:aotopy}
$h$ is an otopy in $\Omega$.
\end{lem}

\begin{proof}
Observe that
\begin{enumerate}
	\item $h^{-1}(0)=k^{-1}(0)\times\{0\}$ is compact,
	\item $k$ is bounded and hence 
	 $\restrictionmap{\id}{\Omega_n}-k$ is compact;
	 in consequence $h(t,x,y)=(x,y)-(x-k(t,x),0)$
	 is of the desired form `identity minus compact'.
\end{enumerate}
Thus $h$ is an otopy in $\Omega$.
\end{proof}

\end{document}